\documentclass[11pt]{amsart}
\usepackage{amssymb}

\newcommand{\Z}{{\mathbb{Z}}}

\newcommand{\N}{{\mathbb{N}}}

\newcommand{\bT}{{\mathbb{T}}}

\newcommand{\bH}{{\mathcal{H}}}

\newcommand{\cL}{{\mathcal{L}}}
\newcommand{\cR}{{\mathcal{R}}}
\newcommand{\cLR}{{\mathcal{LR}}}
\newcommand{\fC}{{\mathfrak{C}}}
\newcommand{\fI}{{\mathfrak{I}}}

\newcommand{\fS}{{\mathfrak{S}}}

\newcommand{\fs}{{\mathfrak{s}}}
\newcommand{\ft}{{\mathfrak{t}}}

\newcommand{\Irr}{{\operatorname{Irr}}}
\renewcommand{\leq}{\leqslant}
\renewcommand{\geq}{\geqslant}
\renewcommand{\atop}[2]{\genfrac{}{}{0pt}{}{#1}{#2}}

\newtheorem{thm}{Theorem}[section]
\newtheorem{lem}[thm]{Lemma}
\newtheorem{cor}[thm]{Corollary}
\newtheorem{prop}[thm]{Proposition}

\theoremstyle{definition}
\newtheorem{exmp}[thm]{Example}
\newtheorem{defn}[thm]{Definition}

\theoremstyle{remark}
\newtheorem{rem}[thm]{Remark}

\begin{document}

\title{Specht modules and Kazhdan--Lusztig cells in type $B_n$}
\author{Meinolf Geck, Lacrimioara Iancu and Christos Pallikaros}

\address{M.G. and L.I.: Department of Mathematical Sciences, King's College,
Aberdeen University, Aberdeen AB24 3UE, Scotland, U.K.}

\email{m.geck@maths.abdn.ac.uk}
\email{l.iancu@maths.abdn.ac.uk}

\address{C.P.: Department of Mathematics and Statistics, University of
Cyprus, P.O. Box 20537, 1678 Nicosia, Cyprus}
\email{pallikar@ucy.ac.cy}

\date{June, 2007}

\begin{abstract} 
Dipper, James and Murphy generalized the classical Specht module theory
to Hecke algebras of type $B_n$. On the other hand, for any choice of a
monomial order on the parameters in type $B_n$, we obtain corresponding
Kazhdan--Lusztig cell modules. In this paper, we show that the Specht
modules are naturally isomorphic to the Kazhdan--Lusztig cell modules 
{\em if} we choose the dominance order on the parameters, as in the 
``asymptotic case'' studied by Bonnaf\'e and the second named author. We 
also give examples which show that such an isomorphism does not exist for 
other choices of monomial orders.
\end{abstract}

\maketitle

\pagestyle{myheadings}
                                                                               
\markboth{Geck, Iancu and Pallikaros}{Specht modules and
Kazhdan--Lusztig cells}

\maketitle

\section{Introduction} \label{sec:intro}

Let $\bH_n$ be the generic Iwahori--Hecke algebra of type $A_{n-1}$ or $B_n$.
For any partition or bipartition $\lambda$ of $n$, we have a corresponding 
Specht module $\tilde{S}^\lambda$, as defined by Dipper--James \cite{DJ0} 
(in type $A_{n-1}$) and Dipper--James--Murphy \cite{DJM} (in type $B_n$).
On the other hand, we have the cell modules arising from the theory of 
Kazhdan--Lusztig cells; see Lusztig \cite{Lusztig83}, \cite{Lusztig03}. 
Now McDonough--Pallikaros \cite{McPa} showed 
that, in type $A_{n-1}$, the Specht modules and Kazhdan--Lusztig cell
modules are naturally isomorphic. The main purpose of this paper is to prove 
an analogous result for type $B_n$. Note that, contrary to the situation
in type $A_{n-1}$, there are many different types of Kazhdan--Lusztig cell 
modules in type $B_n$, depending on the choice of a monomial order on the 
two parameters in type $B_n$. We will show that it is precisely the 
``asymptotic case'' studied in \cite{BI} which yields an isomorphism with 
the Specht modules of Dipper--James--Murphy.  

In Theorem~\ref{thm1}, we show the existence of a canonical isomorphism 
between a Specht module and a Kazhdan--Lusztig left cell module in
the ``asymptotic case'' (where both of them are labelled by the 
appropriate bipartition of $n$). Both the Specht modules and the 
Kazhdan--Lusztig cells have certain standard bases. We show that, for 
a suitable ordering of these bases, the matrix of the canonical 
isomorphism is triangular with $1$ on the diagonal.
Our proof essentially relies on 
the combinatorial description \cite{BI} of the left cells in the
``asymptotic case''. This allows us to determine explicitly (in terms of
reduced expressions of elements) certain distinguished left cells 
for every bipartition of $n$; see Proposition~\ref{prop22}.

In Section~\ref{sec:not}, we give examples which show that the Specht 
modules are {\em not} isomorphic to Kazhdan--Lusztig cell modules
for choices of the monomial order which are different from the 
``asymptotic case''.

\section{Kazhdan--Lusztig bases and cells} \label{sec1}
In this section, we recall the basic definitions concerning Kazhdan--Lusztig
bases and cells, following Lusztig \cite{Lusztig83}, \cite{Lusztig03}. 
We also recall some of the main results of \cite{BI}, \cite{BI2},
\cite{Ge05} concerning the ``asymptotic case'' in type $B_n$. This will
allow us, see Proposition~\ref{prop22}, to describe explicit reduced
expressions for the elements in certain distinguished left cells in type 
$B_n$.

\subsection{Basic definitions} \label{basisdef} $\mbox{}$

In \cite{Lusztig03}, an Iwahori--Hecke algebra with possibly unequal
parameters is defined with respect to an  integer-valued weight function on
$W$. Following a suggestion of Bonnaf\'e \cite{BI2}, we can slightly modify 
Lusztig's definition so as to include the more general setting in 
\cite{Lusztig83} as well. 

Let $\Gamma$ be an abelian group (written 
additively) and let $A={\Z}[\Gamma]$ be the free abelian group with basis 
$\{\varepsilon^\gamma \mid \gamma \in \Gamma\}$. There is a well-defined 
ring structure on $A$ such that $\varepsilon^\gamma \varepsilon^{\gamma'}=
\varepsilon^{\gamma+ \gamma'}$ for all $\gamma,\gamma' \in \Gamma$.  (Hence, 
if $\Gamma=\Z$, then $A$ is nothing but the ring of Laurent polynomials in 
an indeterminate~$\varepsilon$.) We write $1=\varepsilon^0 \in A$.  Given 
$a\in A$ we denote by $a_\gamma$ the coefficient of $\varepsilon^\gamma$, 
so that $a= \sum_{\gamma \in \Gamma} a_\gamma \varepsilon^\gamma$. We say 
that a function 
\[ L \colon W\rightarrow\Gamma\] 
is a weight function if $L(ww')= L(w)+L(w')$ whenever we have $\ell(ww')=
\ell(w)+\ell(w')$ where $\ell\colon W\rightarrow {\N}$ is the usual length 
function. (We denote $\N=\{0,1,2,\ldots\}$.)  
Let $\bH=\bH(W,S,L)$ be the generic Iwahori--Hecke algebra over $A$ with 
parameters $\{v_s \mid s\in S\}$ where $v_s:=\varepsilon^{L(s)}$ for 
$s\in S$. The algebra $\bH$ is free over $A$ with basis $\{T_w\mid w \in
W\}$, and the multiplication is given by the rule
\[ T_sT_w=\left\{\begin{array}{cl} T_{sw} & \quad \mbox{if $\ell(sw)>\ell
(w)$},\\
T_{sw}+(v_s-v_s^{-1})T_w & \quad \mbox{if $\ell(sw)<\ell(w)$},\end{array}
\right.\]
where $s\in S$ and $w\in W$. 

Now assume that there is a total order $\leq$ on $\Gamma$ compatible with 
the group structure. (In the setting of \cite{Lusztig03}, $\Gamma=\Z$ with 
the natural order.) The following definitions will depend on the
choice of this total order. We denote by $A_{\geq 0}$ the set of $\Z$-linear 
combinations of elements $\varepsilon^{\gamma}$ where  $\gamma\geq 0$. 
Similarly, we define $A_{>0}$, $A_{\leq 0}$ and $A_{<0}$. We assume
throughout that $L(s)>0$ for all $s\in S$. 
Having fixed a total order on $\Gamma$, we have a corresponding 
Kazhdan--Lusztig basis $\{C_w \mid w\in W\}$ of $\bH$. The element $C_w$ is 
self-dual with respect to a certain ring involution of $\bH$, and we have
\[ C_w=T_w+\sum_{\atop{y \in W}{y < w}} P_{y,w}^{\,*} \,T_y \in \bH, \]
where $<$ denotes the Bruhat--Chevalley order on $W$ and $P_{y,w}^*\in
A_{<0}$ for all $y<w$ in $W$; see \cite[\S 6]{Lusztig83}. (In the framework 
of \cite{Lusztig03}, the polynomials $P_{y,w}^*$ are denoted $p_{y,w}$ and 
the basis elements $C_w$ are denoted $c_w$.) Given $x,y\in W$, we write 
\[ C_x \,C_y=\sum_{z\in W} h_{x,y,z}\, C_z \qquad \mbox{where 
$h_{x,y,z} \in A$}.\]
We have the following more explicit multiplication rules (see 
\cite[\S 6]{Lusztig83}): for $w \in W$ and $s \in S$, we have
\[ T_s\,C_w = \left\{\begin{array}{ll} \displaystyle{C_{sw}-v_s^{-1}
C_w+\sum_{\atop{z<w}{sz<z}} M_{z,w}^s\, C_z} &\quad \mbox{if
$sw>w$},\\v_s\,C_w &\quad \mbox{if $sw<w$},\end{array}\right.\]
where the elements $M_{z,w}^s \in A$ are determined as in 
\cite[\S 3]{Lusztig83}.

We recall the definition of the left cells of $W$ and the corresponding 
left cell representations of $\bH$ (see \cite{Lusztig83} or \cite{Lusztig03}).
Note again that these depend on the choice of a total order on $\Gamma$.

We write $z \leftarrow_{\cL} y$ if there exists some $s\in S$ such that 
$h_{s,y,z} \neq 0$, that is, $C_z$ occurs in $C_s\, C_y$ (when 
expressed in the $C$-basis).  
Let $\leq_{\cL}$ be the pre-order 
relation on $W$ generated by $\leftarrow_{\cL}$, that is, we have $z 
\leq_{\cL} y$ if there exist elements $z=z_0, z_1,\ldots,z_k=y$ such that 
$z_{i-1} \leftarrow_{\cL} z_i$ for $1\leq i\leq k$.  The equivalence
relation associated with $\leq_{\cL}$ will be denoted by $\sim_{\cL}$ and 
the corresponding equivalence classes are called the {\em left cells} of $W$. 

Similarly, we can define a pre-order $\leq_{\cR}$ by considering
multiplication by $C_s$ on the right in the defining relation. The 
equivalence relation associated with $\leq_{\cR}$ will be denoted by 
$\sim_{\cR}$ and the corresponding equivalence classes are called the 
{\em right cells} of $W$.  We have
\[ x \leq_{\cR} y \quad \Leftrightarrow \quad x^{-1} \leq_{\cL} y^{-1}.\]
This follows by using the anti-automorphism $\flat\colon \bH\rightarrow
\bH$ given by $T_w^\flat=T_{w^{-1}}$; we have $C_w^\flat=C_{w^{-1}}$ ; 
see \cite[5.6]{Lusztig03}. Thus, 
any statement concerning the left pre-order relation $\leq_{\cL}$ has an 
equivalent version for the right pre-order relation $\leq_{\cR}$, via 
$\flat$. 

Finally, we define a pre-order $\leq_{\cLR}$ by the condition 
that $x\leq_{\cLR} y$ if there exists a sequence $x=x_0,x_1,\ldots, x_k=y$ 
such that, for each $i \in \{1,\ldots,k\}$, we have $x_{i-1} \leq_{\cL} x_i$ 
or $x_{i-1}\leq_{\cR}x_i$. The equivalence relation associated with 
$\leq_{\cLR}$ will be denoted by $\sim_{\cLR}$ and the corresponding 
equivalence classes are called the {\em two-sided cells} of $W$.

Each left cell $\fC$ gives rise to a representation of~$\bH$. This is 
constructed as follows (see \cite[\S 7]{Lusztig83}). Let
\begin{align*}
\fI_\fC &=\langle C_y \;(y \in W)\mid  y \leq_{\cL} w \mbox{ for some 
$w \in \fC$}\rangle_A,\\
\hat{\fI}_{\fC} &=\langle C_y \;(y \in W)\mid y \leq_{\cL} w \mbox{ for 
some $w \in \fC$ and } y \not\in\fC\rangle_A.
\end{align*}
These are left ideals in $\bH$. Hence $[\fC]_A=\fI_{\fC}/\hat{\fI}_{\fC}$
is a left $\bH$-module; it is free over $A$ with basis $\{e_w\mid w\in
\fC\}$ where $e_w$ denotes the class of $C_w$ modulo $\hat{\fI}_{\fC}$. 
Explicitly, the action of $\bH$ on $[\fC]_A$ is given by 
\[ C_w.e_x = \sum_{y\in \fC} h_{w,x,y} \, e_y \qquad \mbox{for 
all $x\in \fC$ and $w\in W$}.\]

\subsection{The ``asymptotic case'' in type $B_n$} \label{subbn} $\mbox{}$

Now let $\Gamma=\Z^2$; then $A= {\Z}[\Gamma]$ is nothing but the ring of
Laurent polynomials in two independent indeterminates $V=\varepsilon^{(1,0)}$
and $v=\varepsilon^{(0,1)}$. Let $W=W_n$ be the Coxeter group of type $B_n$ 
($n \geq 2$), with generators, relations and weight function $L \colon W_n 
\rightarrow \Gamma$ given by the following diagram:
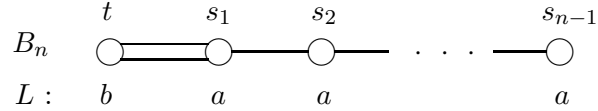
\begin{center}
\begin{picture}(250,50)
\put(  3, 25){$B_n$}
\put(  4, 05){$L$ :}
\put( 40, 25){\circle{10}}
\put( 44, 22){\line(1,0){33}}
\put( 44, 28){\line(1,0){33}}
\put( 81, 25){\circle{10}}
\put( 86, 25){\line(1,0){29}}
\put(120, 25){\circle{10}}
\put(125, 25){\line(1,0){20}}
\put(155, 22){$\cdot$}
\put(165, 22){$\cdot$}
\put(175, 22){$\cdot$}
\put(185, 25){\line(1,0){20}}
\put(210, 25){\circle{10}}
\put( 37, 37){$t$}
\put( 36, 05){$b$}
\put( 76, 37){$s_1$}
\put( 78, 05){$a$}
\put(116, 37){$s_2$}
\put(118, 05){$a$}
\put(203, 37){$s_{n-1}$}
\put(208, 05){$a$}
\end{picture}
\end{center}
where $a,b \in \Gamma$. Let $\bH_n$ be the corresponding generic 
two-parameter Iwahori--Hecke algebra over $A={\Z}[\Gamma]$, where we set
\[ V:=v_t=\varepsilon^b \qquad \mbox{and}\qquad v:=v_{s_1}=\cdots =
v_{s_{n-1}}=\varepsilon^a.\]
(Note that {\it any} Hecke algebra of type $B_n$ can be obtained from 
$\bH_n$ by ``specialisation''; see also Remark~\ref{rem3} below.)
In order to obtain Kazhdan--Lusztig cells and the corresponding cell modules,
we have to specify a total order $\leq$ on $\Gamma$. Note that there are 
infinitely many such total orders: For example, we have all the weighted 
lexicographic orders, given by $(i,j)<(i',j')$ if and only if $xi+yj<xi'+yj'$
or $xi+yj= xi'+yj'$ and $i<i'$, where $x,y$ are fixed positive real numbers. 

Here, we shall take for $\leq$ the lexicographic order on $\Gamma$ such that 
\begin{center}
\fbox{$(i,j)<(i',j')\qquad \stackrel{\text{def}}{\Longleftrightarrow} \qquad
i<i'  \quad \mbox{or} \quad  i=i' \mbox{ and } j<j'$.}
\end{center}
This is the set-up originally considered by Bonnaf\'e--Iancu \cite{BI}; 
it is called the {\em ``asymptotic case''} in type $B_n$. 
We shall need some notation from \cite{BI}. Given $w\in W_n$, we denote by 
$\ell_t(w)$ the number of occurrences of the generator $t$ in a reduced 
expression for $w$, and call this the ``$t$-length'' of $w$. 

The parabolic subgroup $\fS_n=\langle s_1,\ldots,s_{n-1}\rangle$ is
naturally isomorphic to the symmetric group on $\{1,\ldots,n\}$, where
$s_i$ corresponds to the basic transposition $(i,i+1)$. For $1 \leq l
\leq n-1$, we set $\Sigma_{l,n-l}:=\{s_1,\ldots,s_{n-1}\}\setminus
\{s_l\}$.  For $l=0$ or $l=n$, we also set $\Sigma_{0,n}=\Sigma_{n,0}=
\{s_1,\ldots,s_{n-1}\}$.  Then we have the Young subgroup
\[ \fS_{l,n-l}=\langle \Sigma_{l,n-l}\rangle=\fS_{\{1,\ldots,l\}} \times
\fS_{\{l+1,\ldots,n\}}.\]
Let $Y_{l,n-l}$ be the set of distinguished left coset representatives of 
$\fS_{l,n-l}$ in $\fS_n$. We have the parabolic subalgebra $\bH_{l,n-l}=
\langle T_\sigma \mid \sigma \in \fS_{l,n-l}\rangle_A \subseteq \bH_n$.

We denote by $\leq_{\cL,l}$ the Kazhdan--Lusztig (left) pre-order relation
on $\fS_{l,n-l}$ and by $\sim_{\cL,l}$ the corresponding equivalence
relation. The symbols $\leq_{\cR,l}$, $\leq_{\cLR,l}$, $\sim_{\cR,l}$
and $\sim_{\cLR,l}$ have a similar meaning. 

Furthermore, as in \cite[\S 4]{BI}, we set $a_0=1$ and
\[ a_l:=t(s_1t)(s_2s_1t) \cdots (s_{l-1}s_{l-2} \cdots s_1t) \qquad
\mbox{for $l>0$}.\]
Then, by \cite[Prop.~4.4]{BI}, the set $Y_{l,n-l}a_l$ is precisely the 
set of distinguished left coset representatives of $\fS_n$ in $W_n$ whose 
$t$-length equals $l$. Furthermore, every element $w\in W_n$ has a unique 
decomposition
\[w=a_wa_l\sigma_w b_w^{-1} \qquad \mbox{where $l=\ell_t(w)$, $\sigma_w \in
\fS_{l,n-l}$ and $a_w,b_w\in Y_{l,n-l}$};\]
see \cite[4.6]{BI}.  We call this the {\em Clifford normal form} of $w$.

\begin{thm}[Bonnaf\'e--Iancu \protect{\cite[\S 7]{BI}}] \label{mainbi} 
Assume that we are in the ``asymptotic case'' defined above. Let $x,y\in 
W_n$. Then $x\sim_{\cL} y$ if and only if $l:=\ell_t(x)= \ell_t(y)$, 
$b_x=b_y$ and $\sigma_x \sim_{\cL,l} \sigma_y$. 
\end{thm}

\begin{exmp} \label{exp1} Let $l \in \{0,\ldots, n\}$ and $\fC$ be a left
cell of $\fS_{l,n-l}$. Since this group is a direct product, we can write 
$\fC=\fC_1\cdot \fC_2$ where $\fC_1$ is a left cell in $\fS_{\{1,\ldots,l\}}$
and $\fC_2$ is a left cell in $\fS_{\{l+1,\ldots,n\}}$. Now 
Theorem~\ref{mainbi} implies that 
\begin{equation*}
Y_{l,n-l}\, a_l\, \fC\; \mbox{ is a left cell of $W_n$ (in the ``asymptotic
case'')}.\tag{a}
\end{equation*}
Now recall from \cite[4.1]{BI} that $a_l=w_l\sigma_l$ where $w_l$ is
the longest element of the parabolic subgroup $W_l=\langle t,s_1,\ldots,
s_{l-1}\rangle$ (of type $B_l$) and $\sigma_l$ is the longest element
of $\fS_{\{1,\ldots,l\}}$. Since $w_l$ is  central in $W_l$ and 
conjugation with $\sigma_l$ preserves the left cells of $\fS_{\{1,\ldots,
l\}}$, we conclude that $a_l\fC_1 a_l$ is a left cell of $\fS_{\{1,\ldots,
l\}}$, too. Furthermore, $a_l$ commutes with all elements of $\fS_{\{l+1,
\ldots,n\}}$ and so $a_l \fC a_l$ is a left cell of $\fS_{l,n-l}$. 
Applying (a) now yields that 
\begin{equation*}
Y_{l,n-l}\, \fC\,a_l\; \mbox{ is a left cell of $W_n$ (in the ``asymptotic
case'')}.\tag{b}
\end{equation*}
This example will be useful in the proof of Proposition~\ref{prop22} below.
\end{exmp}

\subsection{Bitableaux} \label{bitab} $\mbox{}$

Let $\Lambda_n$ be the set of all bipartitions of $n$. We write such a 
bipartition in the form $\lambda=(\lambda_1|\lambda_2)$ where $\lambda_1$ 
and $\lambda_2$ are partitions such that $|\lambda_1|+ |\lambda_2|=n$. 
For $\lambda \in \Lambda_n$, let $\bT(\lambda)$ be the set of all 
standard $\lambda$-bitableaux. (Whenever we speak of bitableaux, it is 
understood that the filling is by the numbers $1,\ldots,n$.) The 
generalized Robinson--Schensted correspondence of \cite{BI} is a bijection
\[ W_n \stackrel{\sim}{\longrightarrow} \coprod_{\lambda \in \Lambda_n}
\bT(\lambda) \times \bT(\lambda), \qquad w \mapsto (P(w),Q(w)).\]
Thus, to each element $w\in W_n$, we associate a pair of $\lambda$-bitableux 
for some $\lambda \in \Lambda_n$; in this case, we also write $w \leadsto 
\lambda$ and say that $w$ is type $\lambda$. 

The following result provides an explicit combinatorial description of the 
left, right and two-sided cells in the ``asymptotic case'' in type $B_n$.

\begin{thm} \label{mainbi2} Assume we are in the ``asymptotic case'' 
defined in \S \ref{subbn}.  Let $x,y\in W_n$. 
\begin{itemize}
\item[(a)] {\rm (Bonnaf\'e--Iancu \protect{\cite[\S 7]{BI}})} We have 
$x\sim_{\cL} y$ if and only if $Q(x)=Q(y)$. Furthermore, $x \sim_{\cR} y$
if and only if $P(x)=P(y)$.
\item[(b)] {\rm (Bonnaf\'e \protect{\cite[\S3 ]{BI2}})} We have $x 
\sim_{\cLR} y$ if and only if all of $P(x)$, $P(y)$, $Q(x)$ and $Q(y)$ 
have the same shape.
\end{itemize}
\end{thm}

Now let $\fC$ be a left cell of $W_n$. We shall say that $\fC$ is of
type $\lambda \in \Lambda_n$ if the bitableaux $Q(x)$ (where $x \in \fC$)
have shape $\lambda$.

\begin{thm}[Geck \protect{\cite[Theorem~6.3]{Ge05}}] \label{equal} 
Let $\fC$ and $\fC'$ be left cells of $W_n$ (in the ``asymptotic case'')
which have the same type $\lambda \in \Lambda_n$. Then the left cell 
modules $[\fC]_A$ and $[\fC']_A$ are canonically isomorphic. In fact, there 
is a bijection $\fC \leftrightarrow \fC'$ which induces an $\bH_n$-module 
isomorphism $[\fC]_A \stackrel{\sim}{\longrightarrow} [\fC']_A$.
\end{thm}

The above results show that, in order to study the left cell modules of 
$\bH_n$, it is sufficient to exhibit one particular left cell of type 
$\lambda$, for each given $\lambda \in \Lambda_n$. For this purpose, we 
shall need some further combinatorial notions from Dipper--James--Murphy 
\cite[\S 3]{DJM}. 

So let us fix a bipartition $\lambda=(\lambda_1|\lambda_2) \in \Lambda_n$, 
where $l=|\lambda_1|$ and $0 \leq l \leq n$. Let $\ft^\lambda$ be the 
``canonical'' standard bitableau of shape $\lambda$ defined in 
\cite[p.~508]{DJM}.  Thus, $\ft^\lambda$ is a pair consisting of the 
``canonical'' $\lambda_1$-tableau $\ft^{\lambda_1}$ (obtaining by filling 
the rows in order from left to right by the numbers $1,\ldots,l$) and the 
``canonical'' $\lambda_2$-tableau $\ft'^{\lambda_2}$ (obtained by filling 
the rows in order from left to right by the numbers $l+1,\ldots,n$).

The symmetric group $\fS_n$ acts (on the left) on bitableaux by permuting 
the entries. If $\ft$ is any bitableau of shape $\lambda$, denote by $d(\ft)$
the unique element of $\fS_n$ which sends $\ft^\lambda$ to $\ft$. Thus, we 
have $d(\ft).\ft^\lambda=\ft$ for any $\lambda$-bitableau $\ft$. Now let
$\bT^r(\lambda)$ denote the set of all row-standard $\lambda$-bitableaux.
Then 
\[ Y^\lambda:=\{ d(\ft) \mid \ft\in \bT^r(\lambda)\}\]
is the set of distinguished left coset representatives of the parabolic
subgroup $\fS_\lambda$ in $\fS_{n}$; see \cite[p.~509]{DJM}.  Applying
this to the bipartition $((l),(n-l))$, we find that
\[ Y_{l,n-l}=Y^{((l),(n-l))}.\]
Now we also define $\bT^r_l(\lambda)$ to be the set of all $\ft=(\ft_1|
\ft_2)\in \bT^r(\lambda)$ where $\ft_1$ is filled by the numbers $1,\ldots,l$ 
and $\ft_2$ is filled by the numbers $l+1,\ldots,n$. Then, by the same
argument as above,
\[ Y^\lambda_l:=\{ d(\ft) \mid \ft\in \bT^r_l(\lambda)\}\]
is the set of distinguished left coset representatives of the parabolic
subgroup $\fS_\lambda$ inside $\fS_{l,n-l}$. Hence, considering the chain of
parabolic subgroups $\fS_\lambda \subseteq \fS_{l,n-l}\subseteq \fS_n$,
we obtain a decomposition
\[ Y^\lambda=Y_{l,n-l} \cdot Y^\lambda_l\]
where $\ell(yd(\ft))=\ell(y)+\ell(d(\ft))$ for all $y \in Y_{l,n-l}$ 
and $\ft \in \bT^r_l(\lambda)$.

Now we have the following purely combinatorial result.

\begin{lem} \label{lem21} In the above setting, let $\fs\in 
\bT^r(\lambda)$, $\ft\in \bT_l^r(\lambda)$ and $y \in Y_{l,n-l}$ be such 
that $d(\fs)=y\,d(\ft)$. Then $\fs$ is a standard bitableau if and only if
$\ft$ is a standard bitableau.
\end{lem}

\begin{proof} We have $\fs=d(\fs).\ft^\lambda=(yd(\ft)).\ft^\lambda=
y.(d(\ft).\ft^\lambda)=y.\ft$. The permutation $y \in Y_{l,n-l}$ has
the property that $y(i)<y(i+1)$ for $1 \leq i<l$  and $y(i)<y(i+1)$
for $l \leq i <n$. Now it is an easy combinatorial exercise to see that
$\fs$ is standard if and only if $\ft$ is standard; we omit further 
details. 
\end{proof}

\begin{prop} \label{prop22} Let $\lambda=(\lambda_1|\lambda_2) \in \Lambda_n$
and $l=|\lambda_1|$. Let $\sigma_\lambda \in \fS_\lambda$ be the longest
element and $\fC_\lambda$ be the left cell (with respect to the ``asymptotic 
case'') containing $\sigma_\lambda a_l \in W_n$. Then $\fC_\lambda$ has
type $(\lambda_2^*|\lambda_1)$ and we have 
\[\fC_\lambda=\{d(\ft)\,\sigma_\lambda\,a_l\mid\ft\in\bT(\lambda)\},\]
where $\ell(d(\ft) \sigma_\lambda a_l)=\ell(d(\ft))+\ell(\sigma_\lambda
a_l)$ for all $\ft \in \bT(\lambda)$.
\end{prop}

\begin{proof} By relation ($\spadesuit$) in the proof of 
\cite[Prop.~5.4]{GI}, the element $a_l\sigma_\lambda$ has type 
$(\lambda_2^*|\lambda_1)$. Now since $\sigma_\lambda a_l=(a_l
\sigma_\lambda)^{-1}$ it follows that $\sigma_\lambda a_l$ also has type
$(\lambda_2^*|\lambda_1)$. Now, by \cite[Lemma~3.3]{McPa} (extended to the 
direct product of two symmetric groups), the set
\[ \fC:=\{d(\ft)\,\sigma_\lambda\mid \ft \in \bT_l(\lambda)\}\]
is the left cell of $\fS_{l,n-l}$ containing $\sigma_\lambda$, where
$\bT_l(\lambda)$ is the set of all standard $\lambda$-bitableaux in
$\bT^r_l(\lambda)$. Hence, by Example~\ref{exp1}(b), we have
\[ \fC_\lambda=\{y \, d(\ft)\, \sigma_\lambda \, a_l \mid y \in Y_{l,n-l},
\, \ft\in \bT_l(\lambda)\}.\]
Furthermore, $\ell(y\,  d(\ft)\, \sigma_\lambda \, a_l)=
\ell(y\, d(\ft))+\ell(\sigma_\lambda \, a_l)$. Now it remains to use
Lemma~\ref{lem21}.
\end{proof}

\begin{rem} \label{remp22} In the above setting, it is not difficult to 
prove the following related result. Let $x \in W_n$ and 
$l:=\ell_t(x)$. Then we have:

\begin{center} {\em $x \leq_{\cL} \sigma_\lambda a_l\iff x=d(\fs)
\sigma_\lambda a_l$ where $\fs$ is a row-standard $\lambda$-bitableaux.}
\end{center}
This follows from the properties of the Clifford normal form of the elements
in $W_n$ established in \cite[\S7]{BI} and the refinement obtained in
\cite[Theorem~5.11]{Ge05}. As we do not need this result in this 
paper, we omit further details. 
\end{rem}
\section{Specht modules} \label{sec:djm}

We keep the setup of the previous section, where we consider the
Iwahori--Hecke algebra $\bH_n$ of type $B_n$, defined over a polynomial
ring $A={\Z}[V^{\pm 1},v^{\pm 1}]$ in two independent indeterminates. 
We now consider the Specht modules defined by Dipper--James--Murphy
\cite{DJM}. The definition is based on the construction of a new basis 
of $\bH_n$, which is of the form
\[ \{ x_{\fs\ft}=T_{d(\fs)}\, x_\lambda \, T_{d(\ft)^{-1}} \mid 
\lambda \in \Lambda_n \mbox{ and } \fs,\ft \in \bT(\lambda)\}\]
where the element $x_\lambda$ is defined in \cite[4.1]{DJM}; note that
the definition of $x_{\lambda}$ does not rely on the choice of a total 
order on $\Gamma$. (An explicit description of $x_\lambda$ will be given 
in Lemma~\ref{lemxl} below.)

Let $N^\lambda \subseteq \bH_n$ be the $A$-submodule spanned by all 
$x_{\fs\ft}$ where $\fs$ and $\ft$ are standard $\mu$-bitableaux such 
that $\lambda \trianglelefteq \mu$. Here, $\trianglelefteq$ denotes the 
dominance order on bipartitions, which is defined as follows; see 
Dipper--James--Murphy \cite[\S 3]{DJM}: Let $\lambda=(\lambda_1|\lambda_2)$ 
and $\mu=(\mu_1|\mu_2)$ be bipartitions of $n$, with parts
\begin{gather*}
\lambda_1=(\lambda_1^{(1)} \geq \lambda_1^{(2)} \geq \cdots \geq 0),\qquad
\lambda_2=(\lambda_2^{(1)} \geq \lambda_2^{(2)} \geq \cdots \geq 0),\\
\mu_1=(\mu_1^{(1)} \geq \mu_1^{(2)} \geq \cdots \geq 0),\qquad
\mu_2=(\mu_2^{(1)} \geq \mu_2^{(2)} \geq \cdots \geq 0).
\end{gather*}
Then $\lambda \trianglelefteq \mu$ if 
\[ \sum_{i=1}^j \lambda_1^{(i)} \leq \sum_{i=1}^j \mu_1^{(i)} \quad(\forall j)
\quad \mbox{and}\quad |\lambda_1|+ \sum_{i=1}^j \lambda_2^{(i)} \leq 
|\mu_1|+ \sum_{i=1}^j \mu_2^{(i)} \quad(\forall j).\]
By \cite[Cor.~4.13]{DJM}, $N^\lambda$ is a two-sided
ideal of $\bH_n$. Since the basis elements $T_w$ ($w\in W_n)$ are invertible 
in $\bH_n$, we conclude that 
\[ N^\lambda=\sum_{\mu\in \Lambda_n;\, \lambda \trianglelefteq \mu}
\bH_n x_\mu \bH_n.\]
Similarly, we have the two-sided ideal $\hat{N}^\lambda$ spanned by 
all $x_{\fs\ft}$ where $\fs$ and $\ft$ are standard $\mu$-bitableaux
such that $\lambda \triangleleft \mu$ (that is,  $\lambda \trianglelefteq 
\mu$ but $\lambda \neq \mu$). 

\begin{defn}[Dipper--James--Murphy \protect{\cite[Def.~4.19]{DJM}}] 
\label{def0} Let $\lambda\in \Lambda_n$. The corresponding 
{\em Specht module} is defined by 
\[ \tilde{S}^\lambda :=M^\lambda/(M^\lambda \cap \hat{N}^\lambda)
\qquad \mbox{where} \qquad M^\lambda=\bH_n x_\lambda. \] 
By \cite[Theorem~4.20]{DJM}, $\tilde{S}^\lambda$ is free over $A$, with 
standard basis $\{x_\fs \mid \fs\in \bT(\lambda)\}$ where $x_\fs$ denotes
the class modulo $M^\lambda \cap \hat{N}^\lambda$ of the element 
$x_{\fs\ft^\lambda}\in M^\lambda$. 
\end{defn}

Our task will be to identify these Specht modules with certain 
Kazhdan--Lusztig left cells modules. For this purpose, assume from
now on that we have chosen a total order on $\Gamma$ such that we
are in the ``asymptotic case'' defined in \S \ref{subbn}. Our first result, 
which is based on Bonnaf\'e \cite{BI2}, identifies $x_\lambda$ in terms of 
the corresponding Kazhdan--Lusztig basis of $\bH_n$.

\begin{lem} \label{lemxl} Let $\lambda= (\lambda_1 |\lambda_2)\in 
\Lambda_n$ and $l=|\lambda_1|$.  Then 
\[ V^{l}v^{l(l-1)-\ell(\sigma_\lambda)}\,x_\lambda=T_{\sigma_l}\, 
C_{a_l\sigma_\lambda}=C_{\sigma_\lambda a_l} T_{\sigma_l},\] 
where the elements $\sigma_l$, $a_l$ and $\sigma_\lambda$ are 
defined in \S 2. 
\end{lem}

\begin{proof} In \cite[4.1]{DJM}, the element $x_\lambda$ is defined as the 
product of three commuting factors $u_l^+$, $x_{\lambda_1}$, $x_{\lambda_2}$. 
Bonnaf\'e's formula \cite[Prop.~2.5]{BI2} shows that
\[ V^l v^{l(l-1)}\,u_l^+= C_{a_l}T_{\sigma_l}= T_{\sigma_l}C_{a_l}\, .\]
Furthermore, by Lusztig \cite[Cor.~12.2]{Lusztig03}, we have 
$x_{\lambda_1} x_{\lambda_2}=v^{\ell(\sigma_\lambda)}C_{\sigma_\lambda}$.
Finally, by \cite[Prop.~2.3]{BI2}, we have $C_{a_l} C_{\sigma_\lambda}=
C_{a_l \sigma_\lambda}$ and $C_{\sigma_\lambda}C_{a_l}=C_{\sigma_\lambda 
a_l}$. This yields the desired formulas.
\end{proof}

\begin{cor} \label{lem01} Let $\lambda \in \Lambda_n$. Then 
$M^\lambda=\bH_n\, C_{a_l\sigma_\lambda}=\bH_n C_{\sigma_\lambda a_l} 
T_{\sigma_l}$.
\end{cor}

\begin{proof} Clear by Lemma~\ref{lemxl}; just note $v,V$ and 
$T_{\sigma_l}$ are invertible in $\bH_n$. 
\end{proof}

\begin{prop} \label{lem2} Let $\lambda\in \Lambda_n$. Then we have
\begin{align*}
N^\lambda & =\langle C_y \;(y \in W_n)\mid y \leadsto (\nu_1|\nu_2)
\mbox{ where } (\lambda_1|\lambda_2) \trianglelefteq (\nu_2|\nu_1^*)
\rangle_A\tag{a}\\ &\supseteq \langle C_y \; (y \in W_n) \mid y
\leq_{\cLR} a_l\sigma_\lambda\rangle_A,\\
\hat{N}^\lambda & =\langle C_y \;(y \in W_n)\mid C_y \in N^\lambda
\mbox{ and }  y\not\sim_{\cLR} a_l\sigma_\lambda \rangle_A.\tag{b}
\end{align*}
\end{prop}

\begin{proof} (a) The equality is proved in \cite[Theorem~1.5]{GI}.
Now let $y \in W_n$ be such that $y \leq_{\cLR} a_l \sigma_\lambda$.
Assume that $y \leadsto (\mu_1|\mu_2)$. Then Proposition~\ref{prop22} 
and \cite[Prop.~5.4]{GI}  show that $(\mu_1|\mu_2^*) \trianglelefteq 
(\lambda_2^*|\lambda_1^*)$ or, equivalently, $(\lambda_1|\lambda_2) 
\trianglelefteq (\mu_2|\mu_1^*)$. Thus, we have $C_y\in N^\lambda$, as 
required.

(b) Since $\hat{N}^\lambda$ is the sum of all $N^\mu$ where $\mu\in 
\Lambda_n$ and $\lambda \triangleleft \mu$, the equality in (a) also implies 
that 
\[\hat{N}^\lambda=\langle C_y \;(y \in W_n)\mid y \leadsto (\nu_1|\nu_2)
\mbox{ where } (\lambda_1|\lambda_2) \triangleleft (\nu_2|\nu_1^*)
\rangle_A.\]
So (b) follows from the description of the two-sided cells in
Theorem~\ref{mainbi2}(b).
\end{proof}

Now we are ready to construct a canonical homomorphism from a Specht
module to a certain Kazhdan--Lusztig cell module. 

\begin{lem} \label{const1} Let $\lambda=(\lambda_1|\lambda_2) \in \Lambda_n$ 
and $l=|\lambda_1|$. Let $\fC_\lambda$ be the left cell of $W_n$ containing
$\sigma_\lambda a_l$ (with respect to the ``asymptotic case''); see
Proposition~\ref{prop22}. Then there is a unique $\bH_n$-module homomorphism 
$\varphi_\lambda \colon \tilde{S}^\lambda \rightarrow [\fC_\lambda]_A$ which 
sends the class of $x_\lambda \in M^\lambda$ in $\tilde{S}^\lambda$ to the 
class of $C_{\sigma_\lambda a_l} \in \fI_\lambda$ in $[\fC_\lambda]_A$.
\end{lem}

\begin{proof} Recall that $[\fC_\lambda]_A=\fI_\lambda/\hat{\fI}_\lambda$, 
where
\begin{align*}
\fI_\lambda &=\langle C_y \mid y \in W_n \mbox{ such that }
y \leq_{\cL} \sigma_\lambda a_l\rangle_A,\\
\hat{\fI}_\lambda &=\langle C_y \mid y \in W_n \mbox{ such that }
y \leq_{\cL} \sigma_\lambda a_l \mbox{ and } y \not\in
\fC_\lambda\rangle_A.
\end{align*}
We define $\zeta_\lambda:=V^{-(l-1)}\,v^{\ell(\sigma_\lambda)-l(l-1)} \,
T_{\sigma_l}^{-1} \in \bH_n$. (Note that any element of the $T$-basis
is invertible in $\bH_n$.) Then the map 
\[ \rho_\lambda \colon \bH_n \rightarrow \bH_n, \qquad h \mapsto h\,
\zeta_\lambda,\]
(that is, right multiplication by $\zeta_\lambda$) is a left $\bH_n$-module 
isomorphism. By Lemma~\ref{lemxl}, Corollary~\ref{lem01} and the definition
of $\leq_{\cL}$, we have
\[ \rho_\lambda(x_\lambda)=C_{\sigma_\lambda a_l} \qquad \mbox{and}\qquad
\rho_\lambda(M^\lambda)=\bH_n\,C_{\sigma_\lambda a_l}\subseteq\fI_\lambda.\]
Now, by Proposition~\ref{lem2}, we certainly have $\fI_\lambda \cap 
\hat{N}^\lambda \subseteq \hat{\fI}^\lambda$ and so  
\[ \rho_\lambda(M^\lambda \cap \hat{N}^\lambda)\subseteq \bH_n 
C_{\sigma_\lambda a_l} \cap \hat{N}^\lambda \subseteq \fI_\lambda 
\cap \hat{N}^\lambda \subseteq \hat{\fI}_\lambda.\]
Hence, recalling also that $\tilde{S}^\lambda=M^\lambda/M^\lambda \cap
\hat{N}^\lambda$, we obtain a well-defined $\bH_n$-module homomorphism
\[ \varphi_\lambda \colon \tilde{S}^\lambda \rightarrow [\fC_\lambda]_A,
\qquad m+(M^\lambda \cap \hat{N}^\lambda) \mapsto m\zeta_\lambda+
\hat{\fI}_\lambda,\]
having the desired properties. The unicity of $\varphi_\lambda$ is clear 
since $\tilde{S}^\lambda$ is generated, as an $\bH_n$-module, by the 
class of $x_\lambda$.
\end{proof}

Next, we would like to obtain more detailed information about the matrix
of $\varphi_\lambda\colon \tilde{S}^\lambda \rightarrow [\fC_\lambda]_A$ 
with respect to the standard bases of the two modules.  The aim will be to 
show that this matrix is triangular with $1$ on the diagonal; in particular, 
this will show that $\varphi_\lambda$ is an isomorphism. 

Recall that the Specht module $\tilde{S}^\lambda$ has a standard basis 
$\{x_\fs \mid \fs \in \bT(\lambda)\}$; see Definition~\ref{def0}. On the 
other hand, by the definition of cell modules and Proposition~\ref{prop22}, 
$[\fC_\lambda]_A$ has a standard basis $\{e_{d(\fs)\sigma_\lambda a_l}
\mid \fs \in \bT(\lambda)\}$ where $e_{d(\fs)\sigma_\lambda a_l}$ 
denotes the class modulo $\hat{\fI}_\lambda$ of the element
$C_{d(\fs) \sigma_\lambda a_l} \in \fI_\lambda$. So, for any $\ft\in 
\bT(\lambda)$, we write 
\[ \varphi_\lambda(x_\ft)=\sum_{\fs\in \bT(\lambda)} g_{\fs,\ft}\, 
e_{d(\fs)\sigma_\lambda a_l} \quad \mbox{where}\quad g_{\fs,\ft} \in A.\]
Thus, $G_\lambda:=\bigl(g_{\fs,\ft}\bigr)_{\fs,\ft \in \bT(\lambda)}$ is 
the matrix of $\varphi_\lambda$ with respect to the standard bases of 
$\tilde{S}^\lambda$ and $[\fC_\lambda]_A$, respectively.  Now we can 
state the main result of this paper.

\begin{thm} \label{thm1} The map $\varphi_\lambda \colon \tilde{S}^\lambda
\rightarrow [\fC_\lambda]_A$ constructed in Lemma~~\ref{const1} is an
isomorphism. More precisely, the following hold. For any $\fs,\ft \in 
\bT(\lambda)$, we have 
\begin{align*}
g_{\ft,\ft} &=1 \qquad \mbox{for all $\ft \in \bT(\lambda)$},\\
g_{\fs,\ft} &=0 \qquad \mbox{unless $d(\fs)\leq d(\ft)$},\\
g_{\fs,\ft} &\in v^{-1}{\Z}[v^{-1}] \qquad \mbox{if $\fs\neq \ft$};  
\end{align*}
here, $\leq$ denotes the Bruhat--Chevalley order. Thus, the matrix 
$G_\lambda$ has an upper unitriangular shape for a suitable ordering of 
the set $\bT(\lambda)$. 
\end{thm}

\begin{proof}
We begin with the following computation inside the parabolic 
subgroup $\fS_n \subseteq W_n$. Let $\ft\in \bT^r(\lambda)$. By the 
multiplication rules for the Kazhdan--Lusztig basis, $T_{d(\ft)} 
C_{\sigma_\lambda}$ equals $C_{d(\ft) \sigma_\lambda}$ plus a 
${\Z}[v,v^{-1}]$-linear combination of terms $C_x$ where $x \in \fS_n$, 
$x\leq_{\cL,n} \sigma_\lambda$ and $x<d(\ft) \sigma_\lambda$. Now, the 
condition $x \leq_{\cL,n} \sigma_\lambda$ implies that $x$ can be written 
as $x=d(\fs) \sigma_\lambda$ for some $\fs \in \bT^r(\lambda)$ (see, for 
example, \cite[2.9]{McPa}). Then the condition $x=d(\fs)\sigma_\lambda<
d(\ft)\sigma_\lambda$ implies that $d(\fs)<d(\ft)$ (see 
\cite[9.10(f)]{Lusztig03}). Thus, we obtain
\begin{equation*}
T_{d(\ft)}\,C_{\sigma_\lambda}=\sum_{\fs\in \bT^r(\lambda)} 
g_{\fs,\ft}' \, C_{d(\fs)\sigma_\lambda}\qquad \mbox{for any $\ft \in 
\bT^r(\lambda)$},\tag{$*$}
\end{equation*}
where $g_{\fs,\ft}'\in {\Z}[v,v^{-1}]$ for all $\fs,\ft \in \bT^r(\lambda)$;
furthermore, $g_{\ft,\ft}'=1$ and $g_{\fs,\ft}'=0$ unless $d(\fs) 
\leq d(\ft)$ and $d(\fs)\sigma_\lambda \leq_{\cL,n} d(\ft)\sigma_\lambda$. 

To pass from $\fS_n$ to $W_n$, we use the following argument. First note
that $a_l$ is a distinguished right coset representative of $\fS_n$ in $W_n$.
By \cite[Prop.~2.3]{BI2}, we have $C_{\sigma a_l}=C_{\sigma} C_{a_l}$ for 
any $\sigma \in\fS_n$. Hence, multiplying ($*$) on the right by $C_{a_l}$, 
we obtain
\[ T_{d(\ft)}\,C_{\sigma_\lambda a_l}=\sum_{\fs\in \bT^r(\lambda)} 
g_{\fs,\ft}' \, C_{d(\fs)\sigma_\lambda a_l}\qquad \mbox{for any $\ft \in 
\bT^r(\lambda)$}.\]
Now assume that $\ft \in \bT(\lambda)$. Let $s \in \bT^r(\lambda)$ be such 
that $g_{\fs,\ft}'\neq 0$. Then $d(\fs)\sigma_\lambda \leq_{\cL,n}
d(\ft)\sigma_{\lambda}$ and so $d(\fs)\sigma_\lambda a_l \leq_{\cL}
d(\ft)\sigma_\lambda a_l$; see \cite[Prop.~9.11]{Lusztig03}. 
Hence, using Proposition~\ref{prop22}, we find that 
\[ T_{d(\ft)}\,C_{\sigma_\lambda a_l}\equiv \sum_{\fs\in \bT(\lambda)} 
g_{\fs,\ft}' \, C_{d(\fs)\sigma_\lambda a_l}\quad \bmod \quad 
\hat{\fI}_\lambda.\]
Passing to the quotient $\fI_\lambda \rightarrow [\fC_\lambda]_A=
\fI_\lambda/\hat{\fI}_\lambda$, we obtain
\[T_{d(\ft)}.e_{\sigma_\lambda a_l}=\sum_{\fs\in \bT(\lambda)} 
g_{\fs,\ft}' \, e_{d(\fs)\,\sigma_\lambda a_l}\qquad \mbox{for any
$\ft \in \bT(\lambda)$}.\]
Now note that $\varphi_\lambda(x_{\ft})=\varphi_\lambda(
T_{d(\ft)}.\bar{x}_\lambda)=T_{d(\ft)}.\varphi_\lambda(\bar{x}_\lambda)=
T_{d(\ft)}.e_{\sigma_\lambda a_l}$, where $\bar{x}_\lambda$ denotes 
the class of $x_\lambda$ in $\tilde{S}^\lambda$. Thus, we see that 
$g_{\fs,\ft}'=g_{\fs,\ft}$ for all $\fs,\ft \in \bT(\lambda)$. 
Consequently, the coefficients $g_{\fs,\ft}$ have the property that
$g_{\ft,\ft}=1$ and $g_{\fs,\ft}=0$ unless $d(\fs)\leq d(\ft)$. Hence,
for a suitable ordering of the rows and columns, the matrix $G_\lambda$ 
is unitriangular and $\varphi_\lambda$ is an isomorphism.  

It remains to prove that $g_{\fs,\ft}\in v^{-1}{\Z}[v^{-1}]$ for $\fs \neq 
\ft$. We will actually show that $g_{\fs,\ft}'\in v^{-1}{\Z}[v^{-1}]$
for all $\fs,\ft \in \bT^r(\lambda)$ such that $\fs\neq \ft$. This is 
seen as follows. We can invert the equations ($*$) and obtain
\[ C_{d(\ft)\,\sigma_\lambda}=\sum_{\fs\in \bT^r(\lambda)} 
\tilde{g}_{\fs,\ft} \, T_{d(\fs)}\, C_{\sigma_\lambda}\qquad 
\mbox{for any $\ft \in \bT^r(\lambda)$},\]
where the $\tilde{g}_{\fs,\ft}$'s are the entries of the inverse
of the matrix $\bigl(g_{\fs,\ft}'\bigr)_{\fs,\ft \in \bT^r(\lambda)}$. A
comparison with \cite[Prop.~3.3]{Geind} shows that 
\[ \tilde{g}_{\fs,\ft}=p^*_{d(\fs)\sigma_\lambda,d(\ft)\sigma_\lambda}
\in v^{-1}{\Z}[v^{-1}]\qquad \mbox{if $\fs\neq \ft$}.\]
Hence we also have $g_{\fs,\ft}'\in v^{-1}\Z[v^{-1}]$ for $\fs\neq \ft$.
\end{proof}

\begin{rem} \label{rem2} 
Let $\lambda \in \Lambda_n$ and $\fC$ be any left cell such that $\fC$ 
and $\fC_\lambda$ are contained in the same two-sided cell. Then, by 
Theorem~\ref{equal}, $[\fC]_A$ and $[\fC_\lambda]_A$ are canonically 
isomorphic as $\bH_n$-modules. Hence, in combination with 
Theorem~\ref{thm1}, we conclude that $\tilde{S}^\lambda \cong [\fC]_A$. 
Thus, any left cell module of $\bH_n$ is canonically isomorphic to a 
Specht module.
\end{rem}

\begin{rem} \label{rem3} 
The above results also hold for specialized algebras. More precisely, 
let $R$ be any commutative ring (with $1$) and fix two invertible elements
$Q,q\in R$ which admit square roots $Q^{1/2}$ and $q^{1/2}$ in $R$. Then 
we have a unique ring homomorphism $\theta \colon A\rightarrow R$ such 
that $\theta(V)=Q^{1/2}$ and $\theta(v)=q^{1/2}$. We can extend scalars 
from $A$ to $R$ and set
\[ \bH_{n,R}=R \otimes_A \bH_n,\quad \tilde{S}_R^\lambda=R \otimes_A
\tilde{S}^\lambda, \quad [\fC]_R=R\otimes_A [\fC]_A,\]
for any $\lambda \in \Lambda_n$ and any left cell $\fC$ of $W_n$. 
Then $\tilde{S}_R^\lambda$ precisely is the Specht module of the
algebra of type $B_n$ with parameters $Q,q$, as considered by 
Dipper--James--Murphy \cite{DJM}. By Theorem~\ref{thm1} and 
Remark~\ref{rem2}, we have an induced canonical isomorphism 
$\tilde{\varphi}_R \colon \tilde{S}^\lambda_R 
\stackrel{\sim}{\longrightarrow} [\fC]_R$ whenever $\fC$ is in the same
two-sided cell as $\fC_\lambda$.
\end{rem}

\section{Counterexample} \label{sec:not}

Recall that $\bH_n$ is defined over the ring of Laurent polynomials 
$A={\Z}[V^{\pm 1},v^{\pm 1}]$ in two independent indeterminates. In the
previous sections, we considered the Kazhdan--Lusztig cell modules of $\bH_n$
with respect to the ``asymptotic case'' \cite{BI}, that is, assuming that 
the group of monomials $\{V^iv^j \mid i,j \in \Z\}$ is endowed with the
pure lexicographic order where $V^iv^j<v^{i'}v^{j'}$ if and only if $i<i'$ 
or $i=i'$ and $j<j'$. But there are many other monomial orders, each giving 
rise to a Kazhdan--Lusztig basis  of $\bH_n$ and corresponding cell modules. 

The aim of this section is to show that, in general, the Dipper--James--Murphy 
Specht modules $\tilde{S}^\lambda$ cannot be identified with cell modules 
for these other choices of a monomial order. We do this in two ways: (1) by 
a concrete example in type $B_3$ and (2) by a general argument involving
non-semisimple specialisations of $\bH_n$.

\subsection{An example in type $B_3$} \label{expb3} $\mbox{}$

Let $n=3$; then $W_3=\langle t,s_1,s_2\rangle$. Let $\lambda=((1), (2))
\in \Lambda_3$ and consider the corresponding Specht module 
$\tilde{S}^\lambda$.  By Theorem~\ref{thm1}, it is isomorphic to
$[\fC_\lambda]_A$, where $\fC_\lambda$ is a left cell with respect
to the ``asymptotic case''. We have $l=1$ and $\sigma_\lambda a_l=s_2t$. 
Using Proposition~\ref{prop22}, we find that 
\[ \fC_\lambda=\{ s_2t,\quad s_1s_2t,\quad s_2s_1s_2t\}.\]
The corresponding left cell representation $\rho_\lambda \colon \bH_3 
\rightarrow M_3(A)$ is given by 
\begin{gather*} 
T_t \mapsto \begin{pmatrix} V & Vv^{-1}V^{-1}v & Vv^{-2}+V^{-1}v^2\\
0 & -V^{-1} & 0 \\ 0 & 0 & -V^{-1}\end{pmatrix},\\
T_{s_1} \mapsto \begin{pmatrix} -v^{-1} & 0 & 0 \\ 1 & v & 0 \\
0 & 0 & v\end{pmatrix}, \qquad T_{s_2} \mapsto \begin{pmatrix} 
v &  1 & 0 \\ 0 & -v^{-1} & 0 \\ 0 & 1 & v \end{pmatrix}.
\end{gather*} 
Now let us choose a different monomial order on $\{V^iv^j\mid i,j \in \Z\}$, 
namely, the weighted lexicographic order where
\[ V^iv^j <V^{i'}v^{j'}\quad \stackrel{\text{def}}{\Longleftrightarrow} 
\quad i+j<i'+j' \quad \mbox{or} \quad  i+j=i'+j' \mbox{ and } i<i'.\] 
(In particular, we have $v<V<v^2$.)
By an explicit computation, one can show that, in this case, the 
Kazhdan--Lusztig cell modules are all irreducible over $K$ (the field of 
fractions of $A$), in accordance with \cite[Conjecture~A$^+$]{BGIL}. (We are
in the case $r=1$ of that conjecture.) Furthermore, there are precisely 
three left cells $\fC_1, \fC_2,\fC_3$ such that $[\fC_i]_K \cong 
\tilde{S}^\lambda_K$; they are given as follows:
\begin{align*}
\fC_1 &=\{s_1s_2s_1,\quad s_1ts_1s_2s_1,\quad ts_1s_2s_1\},\\
\fC_2 &=\{ s_1s_2s_1t, \quad s_1ts_1s_2s_1t, \quad ts_1s_2s_1t \},\\
\fC_3 &=\{ s_1s_2s_1ts_1,\quad s_1ts_1s_2s_1ts_1,\quad ts_1s_2s_1ts_1 \}.
\end{align*}
The corresponding left cell representations are all identical and given 
by:
\begin{gather*}
T_t \mapsto \begin{pmatrix} -V^{-1} & 0 & 0 \\ 0 & -V^{-1} & 0 \\
1 & Vv^{-1}+V^{-1}v & V \end{pmatrix},\\
T_{s_1} \mapsto \begin{pmatrix} v & 0 & 0 \\ 0 & v & 1 \\ 0 & 0  & -v^{-1}
\end{pmatrix}, \quad T_{s_2} \mapsto \begin{pmatrix} v &Vv^{-2}+V^{-1}v^2 &
0 \\ 0 & -v^{-1} & 0 \\ 0 & 1 & v \end{pmatrix}.
\end{gather*}
Denote this representation by $\rho\colon \bH_3 \rightarrow M_3(A)$. Now 
one checks that $P\rho_\lambda(T_s)=\rho(T_s)P$ for $s\in \{t,s_1,s_2\}$ 
where
\[ P=\begin{pmatrix} 0 & 0 & Vv^{-2}+V^{-1}v^2 \\ 0 & 1 & 0 \\ 1 & 0 & 0 
\end{pmatrix}.\]
Thus, $P$ defines a non-trivial module homomorphism between $\rho_\lambda$
and $\rho$. Since these representations are irreducible over $K$, the matrix
$P$ is uniquely determined up to scalar multiples. But we see that there 
is no scalar $\lambda \in K$ such that $\lambda P \in M_3(A)$ and 
$\det(\lambda P)\in A^\times$. Hence, $\tilde{S}^\lambda$ will not be 
isomorphic to any Kazhdan--Lusztig cell module with respect to the above 
weighted lexicographic order.

\subsection{General cell modules} \label{canbas} $\mbox{}$

Let $k$ be a field and fix an element $\xi\in k^\times$. Let $a,b\in 
\Z_{\geq 0}$ and consider the specialisation $A \rightarrow k$ such that 
$V \mapsto \xi^b$ and $v \mapsto \xi^a$. Let $\bH_{n,k}=k \otimes_A \bH_n$ 
be the corresponding specialized algebra. As in Remark~\ref{rem3}, we also 
have corresponding Specht modules $\tilde{S}^\lambda_k$ for $\bH_{n,k}$. 
Now, for each $\lambda \in \Lambda_n$, there is a certain 
$\bH_{n,k}$-invariant bilinear form $\phi_\lambda \colon \tilde{S}_k^\lambda 
\times \tilde{S}_k^\lambda \rightarrow k$; see \cite[\S 5]{DJM}. Let 
$\mbox{rad}(\phi_\lambda)$ be the radical of that form and set $D^\lambda=
\tilde{S}^\lambda_k/ \mbox{rad}(\phi_\lambda)$. Then $D^\lambda$ is either
$0$ or an absolutely irreducible $\bH_{n,k}$-module; furthermore, we have 
\[ \Irr(\bH_{n,k})=\{D^\mu \mid \mu\in \Lambda^\clubsuit\} \quad \mbox{where} 
\quad \Lambda^\clubsuit=\{\lambda\in \Lambda\mid D^\lambda \neq 0\};\]
see Dipper--James--Murphy \cite[Theorem~6.6]{DJM}. The conjecture in
\cite[8.13]{DJM} about an explicit combinatorial description of 
$\Lambda^\clubsuit$ has recently been proved by Ariki--Jacon \cite{ArJa}.

Now consider the Kazhdan--Lusztig basis $\{C_w\}$ of $\bH_{n,k}$ 
with respect to the weight function $L \colon W_n \rightarrow \Z$
such that $L(t)=b$ and $L(s_i)=a$ for all~$i$. Assume that Lusztig's 
conjectures {\bf (P1)}--{\bf (P15)} in \cite[14.2]{Lusztig03} on 
Hecke algebras with unequal parameters hold. (This is the case, for 
example, in the ``equal parameter case'' where $a=b$; see 
\cite[Chap.~15]{Lusztig03}.) Using these properties, it is shown in 
\cite{mycell} that $\bH_{n,k}$ has a natural ``cellular structure'' in the 
sense of Graham--Lehrer \cite{GrLe}. The elements of the ``cellular basis''
are certain linear combinations of the basis elements $\{C_w\}$. Then,
by the general theory of ``cellular algebras'', for any $\lambda \in
\Lambda_n$, we have a  ``cell module'' $W_k(\lambda)$ for $\bH_{n,k}$ 
and this cell module is naturally equipped with an $\bH_{n,k}$-invariant
bilinear form $g_\lambda\colon W_k(\lambda) \times W_k(\lambda) 
\rightarrow k$. Let $\mbox{rad}(g_\lambda)$ be the radical of that form 
and set $L^\lambda=W_k(\lambda)/ \mbox{rad}(g_\lambda)$. Then, again, 
$L^\lambda$ is either $0$ or an absolutely irreducible 
$\bH_{n,k}$-module; furthermore, we have
\[ \Irr(\bH_{n,k})=\{L^\mu \mid \mu\in \Lambda^\spadesuit\} \quad \mbox{where} 
\quad \Lambda^\spadesuit=\{\lambda\in \Lambda\mid L^\lambda \neq 0\};\]
see Graham--Lehrer \cite[\S 3]{GrLe} and \cite[Example~4.4]{mycell}. 

Given these two settings, it is natural to ask if $\tilde{S}_k^\lambda\cong 
W_k(\lambda)$ and, subsequently, if $\Lambda^\clubsuit=\Lambda^\spadesuit$~?
In the case where $\bH_{n,k}$ is semisimple, it is shown in 
\cite[Example~4.4]{mycell} that $\tilde{S}_k^\lambda \cong W_k(\lambda)$; 
furthermore, by the general theory of ``cellular algebras'' \cite{GrLe} 
and the results in \cite{DJM}, we have $\Lambda^\clubsuit=\Lambda^\spadesuit=
\Lambda$ in this case. However, if $\bH_{n,k}$ is not semisimple, then the 
answer to these questions is negative, as can be seen from the fact that
$\Lambda^\clubsuit\neq \Lambda^\spadesuit$ in general; see \cite{GeJa} 
and the references there.

By \cite[Theorem~2.8]{GeJa} it is true, however, that $\Lambda^\clubsuit=
\Lambda^\spadesuit$ if $b>(n-1)a>0$ which corresponds precisely to the
``asymptotic case'' discussed in this paper. Indeed, by 
\cite[Corollary~6.3]{Ge05}, the basis $\{C_w\}$ of $\bH_{n,k}$ is cellular 
under this assumption on $a,b$, and by Theorem~\ref{thm1}, we have 
$W_k(\lambda) \cong \tilde{S}_k^\lambda$.

\medskip
{\small \noindent {\bf Acknowledgements.} The final form of this paper grew
out of several discussions which the three authors could hold thanks to the
hospitality of various institutions. Part of this work was done while all
three authors enjoyed the hospitality of the Bernoulli Center at the
EPFL Lausanne (Switzerland) in  2005. C.P. thanks the University of 
Aberdeen (Scotland) for an invitation in November 2006; M.G. and L.I.
would like to thank the University of Cyprus at Nicosia for an invitation
in March 2007.}



\begin{thebibliography}{131}

\bibitem{ArJa}
{\sc S.~Ariki and N.~Jacon}, Dipper--James--Murphy's conjecture for Hecke 
algebras of type $B$, preprint; available at {\tt math.RT/0703447}.

\bibitem{BGIL}
{\sc  C.~Bonnaf\'e, M.~Geck, L.~Iancu and T.~Lam},  On domino
insertion and Kazhdan--Lusztig cells in type $B_n$, preprint; available 
at {\tt math.RT/0609279}.

\bibitem{BI}
{\sc C.~Bonnaf\'e and L.~Iancu}, Left cells in type $B_n$ with unequal
parameters, Represent. Theory {\bf 7} (2003), 587--609.

\bibitem{BI2}
{\sc C.~Bonnaf\'e}, Two-sided cells in type $B$ in the asymptotic case,
J. Algebra  {\bf 304} (2006), 216--236.

\bibitem{DJ0}
{\sc R.~Dipper and G.~D.~James}, Representations of Hecke algebras of
general linear groups, {Proc.\ London Math.\ Soc.} {\bf 52} (1986), 20--52.

\bibitem{DJM}
{\sc R.~Dipper, G. D.~James and G. E.~Murphy}, Hecke algebras of type $B_n$
at roots of unity, Proc. London Math. Soc. {\bf 70} (1995), 505--528.

\bibitem{Geind}
{\sc M.~Geck}, On the induction of Kazhdan--Lusztig cells, Bull. London
Math. Soc. {\bf 35} (2003), 608--614.

\bibitem{Ge05}
{\sc M.~Geck}, Relative Kazhdan--Lusztig cells, Represent. Theory {\bf 10}
(2006), 481--524.


\bibitem{mycell}
{\sc M.~Geck}, Hecke algebras of finite type are cellular, Invent. Math.
(2007), to appear.

\bibitem{GI}
{\sc M.~Geck and L.~Iancu},  Lusztig's $a$-function in type $B_n$ in the
asymptotic case. Special issue celebrating the $60$th birthday of
George Lusztig, Nagoya J. Math. {\bf 182} (2006), 199--240.

\bibitem{GeJa}
{\sc M.~Geck and N.~Jacon}, Canonical basic sets in type $B$. Special issue
in honour of Gordon Douglas James, J. Algebra {\bf 306} (2006), 104--127.


\bibitem{GrLe}
{\sc J.~J.~Graham and G.~I.~Lehrer}, Cellular algebras,
Invent.~Math. {\bf 123} (1996), 1--34.



\bibitem{Lusztig83}
{\sc G.~Lusztig}, {Left cells in {W}eyl groups}, {L}ie {G}roup
{R}epresentations, {I} (R.~L. R.~Herb and J.~Rosenberg, eds.), Lecture Notes
in Math., vol.  1024, Springer-Verlag, 1983, pp.~99--111.

\bibitem{Lusztig03}
{\sc G.~Lusztig}, Hecke algebras with unequal parameters, CRM Monographs
Ser.~{\bf 18}, Amer. Math. Soc., Providence, RI, 2003.

\bibitem{McPa}
{\sc T.~P.~McDonough and C.~A.~Pallikaros}, On relations between the
classical and the Kazhdan--Lusztig representations of symmetric groups
and associated Hecke algebras, J. Pure and Applied Algebra {\bf 203}
(2005), 133--144.
\end{thebibliography}
\end{document}